\documentclass{amsart}
\usepackage{enumerate}
\usepackage{amssymb}
\usepackage{amsmath}
\usepackage[OT2,OT1]{fontenc}
\newcommand\cyr{%
\renewcommand\rmdefault{wncyr}%
\renewcommand\sfdefault{wncyss}%
\renewcommand\encodingdefault{OT2}%
\normalfont \selectfont} \DeclareTextFontCommand{\textcyr}{\cyr}

\newcommand{\snug}{\unskip\kern-\mathsurround}
\newcommand{\ad}{{\rm ad}}
\newcommand{\tg}{{\rm tg}}
\newcommand{\gr}{{\rm gr}}
\newcommand{\cd}{{\rm cd}}
\newcommand{\Lie}{{\rm L}}

\newcommand{\Z}{{\mathbb Z}}
\newcommand{\F}{{\mathbb F}}
\newcommand{\N}{{\mathbb N}}
\newcommand{\Q}{{\mathbb Q}}
\newtheorem{theorem}{Theorem}[section]

\newtheorem{proposition}[theorem]{Proposition}
\newtheorem{corollary}[theorem]{Corollary}

\newtheorem{definition}[theorem]{Definition}

\begin{document}
\title{Mild Pro-$2$-Groups and $2$-Extensions of $\Q$ with Restricted Ramification}
\date{October 13, 2009}
\author{John Labute}
\email{labute@math.mcgill.ca}
\address{Department of Mathematics and Statistics, McGill University, Burnside Hall, 805 Sherbrooke Street West, Montreal QC H3A 2K6, Canada}
\author[J\'{a}n Min\'{a}\v{c}]{J\'an Min\'a\v{c}}
\email{minac@uwo.ca}
\address{Department of Mathematics, Middlesex College, \
University of Western Ontario, London, Ontario \ N6A 5B7 \ CANADA}
\thanks{$^\dag$Research supported in part by NSERC grants 4344-06 and R0370A01.}
\subjclass{11R34, 20E15, 12G10, 20F05, 20F14, 20F40}
\begin{abstract}
Using the mixed Lie algebras of Lazard, we extend the results of
the first author on mild groups to the case $p=2$. In particular,
we show that for any finite set $S_0$ of odd rational primes we
can find a finite set $S$ of odd rational primes containing $S_0$
such that the Galois group of the maximal $2$-extension of $\Q$
unramified outside $S$ is mild. We thus produce a projective
system of such Galois groups which converge to the maximal
pro-$2$-quotient of the absolute Galois group of $\Q$ unramified
at $2$ and $\infty$. Our results also allow results of Alexander
Schmidt on pro-$p$-fundamental groups of marked arithmetic curves
to be extended to the case $p=2$ over a global field which is
either a function field of characteristic $\ne2$ or a totally
imaginary number field.
\end{abstract}
\maketitle

\hfill {\it \`{A} Serre\ \ \ \ \ \ }
\section{Introduction}
In this paper we extend the theory of mild pro-$p$-groups
developed in \cite{La} to the case $p=2$. In particular, we obtain
the following result which is the missing ingredient in extending
the results of Alexander Schmidt in \cite{Sch} to the case $p=2$
over a global field which is either a function field of
characteristic $\ne2$ or a totally imaginary number field. Let
$H^i(G)=H^i(G,\Z/p\Z)$.
\medskip

\begin{theorem}\label{main1}
Let $G$ be a finitely generated pro-$p$-group. If $H^2(G)\ne 0$
and $H^1(G)=U\oplus V$ with the cup-product trivial on $U\times U$
and mapping $U\otimes V$ surjectively onto $H^2(G)$ then $G$ is
mild.
\end{theorem}
\medskip

For $p\ne2$, Theorem~\ref{main1} is a reformulation by Schmidt of
a criterion for the mildness of a pro-$p$-group that was proven in
\cite{La}. We will show that mild pro-$p$-groups are also of
cohomological dimension $2$ when $p=2$. To prove our results we
have to further develop the theory of certain mixed Lie algebras
of Lazard~\cite{Laz}.

If $S$  is a finite set of odd rational primes we let $G_S(2)$ be
the Galois group of the maximal $2$-extension of $\Q$ unramified
outside $S$.
\medskip

\begin{theorem}\label{main2}
If $S_0$ is a finite set of odd rational primes there is a finite
set $S$ of odd rational primes containing $S_0$ such that $G_S(2)$
is mild.
\end{theorem}
\medskip

Although the study of Galois groups of number fields with
restricted ramification can be traced already to work of L.
Kronecker and others in the 19-th century, the formal modern
foundations were laid out by I.R. \v Shafarevi\v c. His work was
influenced by geometrical considerations of finite coverings of
Riemann surfaces ramified in a given finite set of primes, class
field theory and a deep understanding of the Galois groups of
local fields. His papers \cite{Sha1}, \cite{Sha2} as well as his
paper with E.S. Golod \cite{GoSha} demonstrated the extraordinary
power of his vision. Koch's monograph \cite{Ko}, first published
in 1970, summarized the important contributions to the subject.
For example, information of the cohomological dimension of
$G_S(p)$ was obtained when $p$ was odd and in $S$. When $p$ was
not in $S$, nothing was known about $G_S(p)$, other that it could
be infinite by the work of Golod and Shafarevich, until the recent
work of the first author \cite{La} where it was shown that for $p$
odd this group was of cohomological dimension 2 for certain $S$.
The more difficult case $p=2$ was left open. This work finally
extends these results to the case $p=2$.

\section{Mixed Lie Algebras}

Let $G$ be a pro-$2$-group and let $G_n \ (n\ge1)$ be the $n$-th
term of the lower $2$-central series of $G$. We have
$$
G_1=G,\ \ G_{n+1}=G_n^2[G,G_n]
$$
where, for subgroups $H,K$ of $G$, $[H,K]$ is the closed subgroup
generated by the commutators $[h,k]=h^{-1}k^{-1}hk$ with $h\in H,
k\in K$ and $H^2$ is the subset of squares $h^2$ of elements of
$H$. Let $L(G)$ be the Lie algebra associated to the lower
$2$-central series of $G$. We have
$$
L(G)=\oplus_{n\ge1}L_n(G)
$$
where $L_n(G)=G_n/G_{n+1}$ is denoted additively. This defines
$\Lie(G)$ as a graded vector space over $\F_2$. If $l_n$ is the
canonical homomorphism $G_n\rightarrow L_n(G)$, the Lie bracket
$[\xi,\eta]$ of $\xi=l_m(x)$, $\eta=l_n(y)$ is $l_{m+n}([x,y])$.
To the homogeneous element $\xi=l_n(x)$ we associate the
homogeneous element $P\xi=l_{n+1}(x^2)$. If $\xi,\eta\in L_n(G)$
then
$$
P(\xi+\eta)=\begin{cases}
 P\xi+P\eta & \text{if $n>1$},\\
 P\xi+P\eta+[\xi,\eta] & \text{if $n=1$}.
\end{cases}
$$

If $\xi\in L_m(G)$, $\eta\in L_n(G)$ we have
$$
[P\xi,\eta]=\begin{cases}
 P[\xi,\eta] & \text{if $m>1$},\\
 P[\xi,\eta]+[\xi,[\xi,\eta]] & \text{if $m=1$}.
 \end{cases}
$$
Thus the operator $P$ defines a mixed Lie algebra structure on
$L(G)$ in the terminology of Lazard, cf. \cite{Laz}, Ch.2, \S1.2.
The operator $P$ extends to a linear operator on the Lie algebra
$$
L^+(G)=\oplus_{n>1}\Lie_n(G).
$$
It follows that $L ^+(G)$ is a module over the polynomial ring
$\F_2[\pi]$ where $\pi u=P(u)$.

If $A=\sum_{n\ge0}A_n$ is a graded associative algebra over the
graded algebra $\F_2[\pi]$, where multiplication by $\pi$ on
homogeneous elements increases the degree by $1$, then
$A_+=\sum_{n>0}A_n$ has the structure of a mixed Lie algebra where
$$
P\xi=\begin{cases}
 \pi\xi & \text{if $\xi$ is of degree $>1$},\\
 \pi\xi+\xi^2 & \text{if $\xi$ is of degree $1$}.
 \end{cases}
$$
Every mixed Lie algebra $\mathfrak g$ has an enveloping algebra
$U_{{\rm mix}}(\mathfrak g)$. This is graded associative algebra
$U$ over $\F_2[\pi]$ together with and a mixed Lie algebra
homomorphism $f$ of $\mathfrak g$ into $U_+$ such that, for every
graded associative algebra $B$ over $\F_2[\pi]$  and mixed Lie
algebra homomorphism $\varphi_0$ of $\mathfrak g$ into $B_+$,
there is a unique algebra homomorphism $\varphi$ of $U$ into $B$
satisfying $\varphi\circ f=\varphi_0$. The existence of $U_{{\rm
mix}}(\mathfrak g)$ is proven in \cite{Laz}, Th. 1.2.8. It is also
shown there that the canonical mapping of $\mathfrak{g}$ into
$U({\mathfrak{g}})$ is injective; this fact is referred to as the
Birkhoff-Witt Theorem for mixed Lie algebras. If
$X=\{x_1,\ldots,x_d\}$ is a weighted set, the enveloping algebra
of the free mixed Lie algebra $L_{\rm mix}(X)$ on the weighted set
$X$ is the free associative algebra $A(X)$ over $\F_2[\pi]$ on
$X$. Indeed, giving a mixed Lie algebra homomorphism $f: L_{\rm
mix}(X)\rightarrow B_+$ is the same as giving a graded map of $X$
into $B_+$ which is the same as giving a homomorphism of the
graded algebra $A(X)$ into $B$. It is now a straight-forward
argument to verify the following Proposition.
\medskip

\begin{proposition}
If $0\rightarrow\mathfrak r\rightarrow\mathfrak
g\rightarrow\mathfrak h\rightarrow 0$ is an exact sequence of
mixed Lie algebras, we have
$$
U_{\rm mix}(\mathfrak h)=U_{\rm mix}(\mathfrak g)/\mathfrak R
$$
where $\mathfrak R$ is the ideal of $U_{\rm mix}(\mathfrak g)$
generated by the image of $\mathfrak r$.
\end{proposition}
\medskip

Let $X=\{x_1,\ldots,x_d\}$ be a set and let $F=F(X)$ be the free
pro-$2$-group on $X$. The completed group algebra
$\Lambda=\Z_2[[F]]$ over the $2$-adic integers $\Z_2$ is
isomorphic to the Magnus algebra of formal power series in the
non-commuting indeterminates $X_1,\ldots,X_d$ over $\Z_2$.
Identifying $F$ with its image in $\Lambda$, we have $x_i=1+X_i$
(cf. \cite{Se}, Ch. I, \S 1.5).

The lower $2$-cental series of $F$ can be obtained by means of a
valuation on $\Lambda$. More generally, if $\tau_1,\ldots,\tau_d$
are integers $>0$, we define a valuation $w$ in the sense of
Lazard by setting
$$
w(\sum_{i_1,\ldots,i_k}a_{i_1,\ldots,i_k}X_{i_1}\cdots
X_{i_k})=\inf_{i_1,\ldots,i_k}(v(a_{i_1,\ldots,i_k})+\tau_{i_1}+\cdots+\tau_{i_k}),
$$
where $v$ is the $2$-adic valuation of $\Z_2$ with $v(2)=1$. Let
$\Lambda_n=\{u\in A\mid w(u)\ge n\}.$ Then $(\Lambda_n)_{n\ge0}$
is a filtration of $\Lambda$ by ideals and the associated graded
algebra $\gr(\Lambda)$ is a graded algebra over the graded ring
$\F_2[\pi]=\gr(\Z_2)$ with $\pi$ the image of $2$ in
$2\Z_2/4\Z_2$. If $\xi_i$ is the image of $X_i$ in
$\gr_{\tau_i}(\Lambda)$ then $\gr(\Lambda)$ is the free
associative $\F_2[\pi]$-algebra on $\xi_1,\ldots,\xi_d$ with a
grading in which $\xi_i$ is of degree $\tau_i$ and multiplication
by $\pi$ increases the degree by $1$. The Lie subalgebra $L$ of
$\gr(\Lambda)$ generated by the $\xi_i$ is the free mixed Lie
algebra over $\F_2[\pi]$ on $\xi_1,\ldots,\xi_d$ by the
Birkhoff-Witt Theorem. Note that when $\tau_i=1$ for all $i$ we
have $\Lambda_n=I^n$, where $I$ is the augmentation ideal
$(2,X_1,\ldots,X_d)$ of $\Lambda$.

For $n\ge1$, let $F_n=(1+\Lambda_n)\cap F$ and for $x\in F$ let
$\omega(x)=w(x-1)$ be the filtration degree of $x$. Then $(F_n)$
is a decreasing sequence of closed subgroups of $F$ with the
following properties:
$$
F_1=F,\ [F_n,F_k]\subseteq F_{n+k},\ F_n^2\subseteq F_{n+1}.
$$
It is called the $(x,\tau)$-filtration of $F$. Such a sequence of
subgroups of a pro-$2$-group $G$ is called a $2$-central series of
$G$. If $\tau_i=1$ for all $i$ then $(F_n)$ is the lower
$2$-central series of $F$.

If $(G_n)$ is a $2$-central series of $G$, let
$\gr_n(G)=G_n/G_{n+1}$ with the group operation denoted
additively. Then $\gr(G)=\oplus_{n\ge1}\gr_n(G)$ is a graded
vector space over $\F_2$ with a bracket operation $[\xi,\eta]$
which is defined for $\xi\in G_n$, $\eta\in G_k$ to be the image
in $\gr_{n+k}(F)$ of $[x,y]$ where $x,y$ are representatives of
$\xi,\eta$ in $\gr_n(G),\gr_k(G)$ respectively. Under this bracket
operation, $\gr(G)$ is a Lie algebra over $\F_2$. The mapping
$x\mapsto x^2$ induces an operator $P$ on $\gr(G)$ sending
$\gr_n(G)$ into $\gr_{n+1}(G)$. For homogeneous $\xi,\eta$ of
degree $m,n$ respectively, we have
\begin{align*}
P(\xi+\eta)&=P(\xi)+P(\eta)+[\xi,\eta]\ {\rm if }\ m=n=1,\\
P(\xi+\eta)&=P(\xi)+P(\eta)\ {\rm if }\ m=n>1,\\
[P(\xi),\eta]&=P([\xi,\eta])+[\xi,[\xi,\eta]]\ {\rm if }\ m=1,\\
[P(\xi),\eta]&=P([\xi,\eta])\ {\rm if }\ m>1.
\end{align*}
Hence $\gr(G)$ is a mixed Lie algebra.

 In the case $F=F(X)$ and $F_n=(1+\Lambda_n)\cap F$, the mapping $x\mapsto
x-1$ induces an injective Lie algebra homomorphism of $\gr(F)$
into $\gr(\Lambda)$. Identifying $\gr(F)$ with its image in
$\gr(\Lambda)$, we have $P(\xi)=\pi\xi$ unless $\xi\in\gr_1(F)$ in
which case
$$
P(\xi)=\xi^2+\pi\xi.
$$
\medskip
The Lie algebra $\gr(F)$ is the smallest $\F_2$-subalgebra of
$\gr(\Lambda)$ which contains $\xi_1,\ldots\xi_d$ and is stable
under $P$. To see this, let $X_n$ be the set of elements $x_i$
with $\tau_i=n$ and define subsets $T_n$ inductively as follows:
$T_1=X_1$ and, for $n>1$, $T_n=T_n'\cup T_n''$ where
$$
T_n'=\{x^2\mid x\in T_{n-1}\},\quad T_n''= X_n\cup\{[x,y]\mid x\in
T''_r,y\in T''_s,\ r+s=n\}.
$$
If $F_n'$ is the closed subgroup of $F$ generated by the $T_k$
with $k\ge n$, then $(F'_n)$ is a $2$-central series of $F$ (cf.
\cite{Laz}, \S1.2). If $\gr'(F)$ is the associated graded
Lie-algebra, the inclusions $F_n'\subseteq F_n$ induce a mixed Lie
algebra homomorphism $\gr'(F)\rightarrow \gr(F)$. We obtain a
sequence of mixed Lie algebra homomorphisms
$$
L_{\rm mix}(X)\rightarrow \gr'(F)\rightarrow \gr(F)\rightarrow
\gr(\Lambda),
$$
where the homomorphism $L_{\rm mix}(X)\rightarrow \gr'(F)$ sends
$\xi_i$ to $\xi_i'$, the image of $\xi_i$ in $\gr_{\tau_i}'(F)$,
and hence is surjective since the $\xi'_i$ generate $\gr'(F)$ as a
mixed Lie algebra over $\F_2[\pi]$. The composite of these
homomorphisms sends $\xi_i$ to $\xi_i$ and hence is injective.
Thus $\gr'(F)\rightarrow \gr(F)$ is injective from which it
follows inductively that $F_n'=F_n$ for all $n$. Hence we obtain
that $\gr(F(X))=L_{\rm mix}(X)$. The above $2$-filtration $(F_n)$
of $F$ is called the $(x,\tau)$-filtration of $F$. If $\tau_i=1$
for all $i$ then $(F_n)$ is the lower $2$-central series of $F$.
Thus we have shown the following result.
\medskip

\begin{theorem}
If $L(F(X))$ is the Lie algebra associated to the
$(x,\tau)$-filtration of the free pro-$2$-group $F(X)$ on the
weighted set $X=\{x_1,\ldots,x_d\}$, with $x_i$ of weight
$\tau_i$, then
 $L(F(X))=L_{\rm mix}(X)$, the free mixed Lie algebra on
 $X=\{\xi_1,\ldots,\xi_d\}$,
 where $\xi_i$ is the image of $x_i$ in $L_{\tau_i}(F(X))$.
\end{theorem}
\medskip

\begin{theorem}\label{L+}
$L^+(X)$ is a free Lie algebra over $\F_2[\pi]$. If
$\xi_1,\ldots,\xi_m$ are the elements of $X$ of weight $1$ then,
as a free Lie algebra, $L^+(X)$ has a basis $Y$ consisting of
\begin{enumerate}[{\rm (1)}]
\item the $\binom{m+1}{2}$ elements
$$
P\xi_1,\ldots,P\xi_m,\  [\xi_i,\xi_j]\ (1\le i<j\le m),
$$
\item the elements
$$\xi_{m+1},\ldots,\xi_d,\ [\xi_i,\xi_j]\ (1\le i\le m,\ m+1\le
j\le d),
$$
\item for $3\le k$, the $(k-1)\binom{m}{k-1}$ commutators
$$
\ad(\xi_{i_1})\ad(\xi_{i_2})\cdots\ad(\xi_{i_{k-3}})\ad(\xi_j)^2(\xi_{i_{k-2}}),
$$
where $m\ge i_1>i_2>\cdots>i_{k-2}\ge1$, $1\le j\le m$, $j\ne
i_1,\ldots,i_{k-2}$,
\item for $3\le k$, the $(k-1)\binom{m}{k}$ commutators
$$
\ad(\xi_{i_1})\ad(\xi_{i_2})\cdots\ad(\xi_{i_{k-2}})\ad(\xi_{i_{k-1}})(\xi_{i_k}),
$$
where $m\ge i_1>i_2>\cdots>i_{k-1}\ge1$, $i_{k-1}<i_k\le m$,
$i_k\ne i_1,\ldots,i_{k-2}$,
\item for $3\le k$, the $\binom{m}{k-1}(d -m)$ commutators
$$
\ad(\xi_{i_1})\ad(\xi_{i_2})\cdots\ad(\xi_{i_{k-2}})\ad(\xi_{i_{k-1}})(\xi_{i_k}),
$$
where $m\ge i_1>i_2>\cdots>i_{k-1}\ge1$, $i_k>m$.
\end{enumerate}
If $A=A(X)$ is the free associative $\F_2[\pi]$-algebra on $X$ and
$B$ is the subalgebra of $A$ generated by\, $Y$ then $B$ is the
free associative algebra over $\F_2[\pi]$ on the weighted set $Y$.
Moreover, $A$ is a free $B$-module with basis
$\xi_1^{e_1}\cdots\xi_s^{e_s}$ ($e_i=0,1$).
\end{theorem}
\begin{proof}
Let $A$ be the free associative algebra on
$X=\{\xi_1,\ldots,\xi_d\}$ over $\F_2[\pi]$ and let $\bar L$ be
the Lie subalgebra over $\F_2$ generated by $X$. Then $\bar L$ is
the free Lie algebra over $\F_2$ generated by $X$. If $L=L_{\rm
mix} (X)$ we have
\begin{align*}
L_1&=\bar L_1=\sum_{i=1}^m\F_2\xi_i,\\
L_n&=\pi^{n-2}\sum_{i=1}^m\F_2P\xi_i+\pi^{n-2}\bar
L_2+\cdots+\pi\bar L_{n-1}+\bar L_n \ (n\ge2).
\end{align*}
Let $Z$ be a homogeneous basis of $\bar{L}$ containing $X$ with
$\xi_1,\ldots,\xi_m$ the elements of $Z$ of degree $1$. If $Z^+$
is the set of elements of $Z$ of degree $>1$ then
$$
Z^*=\{P\xi_1,P\xi_2,\ldots,P\xi_m\}\cup Z^+
$$
is an $\F_2$-basis for $L^+$ modulo $\pi L^+$ and hence is an
$\F_2[\pi]$-basis for the free $\F_2[\pi]$-module $L^+$. If
$Z=\{\eta_i\mid i\ge1\}$ is linearly ordered so that
$\eta_i\le\eta_{i+1}$ and ${\rm degree}(\eta_i)\le {\rm
degree}(\eta_{i+1})$ then, by the Birkhoff-Witt theorem for Lie
algebras over $\F_2$, the elements
$$
\eta^\alpha=\prod_{i\ge1}\eta_i^{\alpha_i},
$$
where $\alpha=(\alpha_i)_{i\ge1}$ with $\alpha_i=0$ for almost all
$i$, form a $\F_2$-basis of $\bar{A}=A/\pi A$, the enveloping
algebra of $\bar{L}$. It follows that the elements
$$
\prod_{i=1}^m\eta_i^{\beta_i}\prod_{i=1}^m\eta_i^{2\gamma_i}\prod_{i>m}\eta_i^{\alpha_i},
$$
where $\beta_i=0,1$ and $\gamma_i,\alpha_i\in \N$, are also an
$\F_2$-basis of $\bar{A}$. Note that, in our convention, $0\in\N$.
Hence the elements
$$
\prod_{i=1}^m\eta_i^{\beta_i}\prod_{i=1}^m
P\eta_i^{\gamma_i}\prod_{i>m}\eta_i^{\alpha_i},
$$
where $\beta_i=0,1$ and $\gamma_i,\alpha_i\in \N$, are a
$\F_2[\pi]$-basis for $A$. In particular, the elements
$$
\prod_{i=1}^mP\eta_i^{\gamma_i}\prod_{i>m}\eta_i^{\alpha_i},
$$
where $\gamma_i,\alpha_i\in \N$, are an $\F_2[\pi]$-basis for the
$\F_2[\pi]$-subalgebra $B$ of $A$ generated by $Z^*$. This implies
that $A$ is a free $B$-module with basis
$$
\xi_1^{i_1}\cdots\xi_m^{i_m}\quad (i_k=0,1).
$$

Let $a_n$ be the number of elements of $Z$ of degree $n$. Then
$$
\prod_{n\ge1}(1-t^n)^{-a_n}=\frac{1}{1-\sum_im_it^{e_i}},
$$
where $e_1<e_2<\cdots<e_r$ are the possible values of the
$\tau_i=\deg(\xi_i)$ and $m_i$ is the number of $j$ with
$\tau_j=e_i$. We can rewrite this equation in the form
$(1+t)^mP(t)=(1-\sum_im_it^{e_i})^{-1}$ where

\begin{align*}
P(t)&=(1-t^2)^{-m}\prod_{n\ge2}(1-t^n)^{-a_n}\\
&=(1-t^2)^{-(a_2+m)}\prod_{n\ge3}(1-t^n)^{-a_n}\\
&=\frac{1}{1-(c_2t^2+c_3t^3+\cdots+c_{m+1}t^{m+1}+\sum_{k\ge1}q_k(t))},
\end{align*}
where
\begin{align*}
c_k&=(k-1)\binom{m+1}{k}
\\&=(k-1)\binom{m}{k-1}+(k-1)\binom{m}{k},\\
q_k(t)&=\sum_{j\ge2}\binom{m}{k-1}m_jt^{k-1+e_j}.
\end{align*}
The power series $P(t)$ is the Poincar\'e series of $\bar B=B/\pi
B$; the Poincar\'e series of $B$ is $P(t)/(1-t)$.

To show that the elements of $Y$ generate $L^+$ it suffices to
show that they generate $L^+$ as a vector space over $\F_2$ modulo
$\pi L^++[L^+,L^+]$. For $k>2$, we have $L^+_k=\bar L_k$ modulo
$\pi L^+$. For $k\ge2$, every element of $\bar L_k$ can be
uniquely written modulo $[\bar L,\bar L]$ as a linear combination
of the sequence $S$ of elements of the form
$$
\ad(\xi_{i_1})\ad(\xi_{i_2})\cdots\ad(\xi_{i_{k-2}})\ad(\xi_{i_{k-1}})(\xi_{i_k})
$$
with $d\ge i_1\ge i_2\ge\cdots\ge i_{k-1}\ge 1$ and $i_{k-1}<i_k$.
Modulo $\pi L^+$ we have
\begin{align*}
[P(\xi_i),P(\xi_j)]&=\ad(\xi_i)\ad(\xi_j)^2(\xi_i),\\
[P(\xi_i),u]&=\ad(\xi_i)^2(u)\ {\rm if}\ u\in L^+
\end{align*}
and $\ad(\xi_i)\ad(\xi_j)(u)=\ad(\xi_j)\ad(\xi_i)(u)$ modulo
$[L^+,L^+]$ if $u\in L^+$. If follows that the only terms of the
sequence $S$ which possibly do not lie in $\pi L^++[L^+,L^+]$ are
the terms of the subsequence $T$ of elements of the form
$$
{\rm (A)}\qquad\ \ \ \
\ad(\xi_{i_1})\ad(\xi_{i_2})\cdots\ad(\xi_{i_{k-2}})\ad(\xi_{i_{k-1}})(\xi_{i_k})
$$
with $m\ge i_1> i_2>\cdots> i_{k-1}\ge 1$ and $i_{k-1}<i_k$, or of
the form
$$
{\rm
(B)}\qquad\ad(\xi_{i_1})\ad(\xi_{i_2})\cdots\ad(\xi_{i_{k-3}})\ad(\xi_{i_{k-2}})^2(\xi_{i_{k-1}})
$$
with $m\ge i_1>i_2>\cdots>i_{k-2}\ge1$, $i_{k-2}<i_{k-1}\le m$, or
of the form
$$
{\rm (C)}\qquad\ \ \
\ad(\xi_{i_1})\ad(\xi_{i_2})\cdots\ad(\xi_{i_{k-2}})\ad(\xi_{i_{k-1}})(\xi_{i_k})
$$
with $m\ge i_1> i_2>\cdots\ge i_{k-1}\ge 1$ and $i_{k-1}<i_k=i_1$.
Working modulo $\pi L^++[L^+,L^+]$, this last element is equal to
$$
\ad(\xi_{i_2})\cdots\ad(\xi_{i_{k-2}})\ad(\xi_{i_1})\ad(\xi_{i_{k-1}})(\xi_{i_k})=
\ad(\xi_{i_2})\cdots\ad(\xi_{i_{k-2}})\ad(\xi_{i_1})^2(\xi_{i_{k-1}})
$$
which is an element in the family (3) in the statement of the
theorem. Using the identity
$$\ad(x)\ad(y)^2\ad(z)=\ad(z)\ad(y)^2\ad(x)\ \ ({\rm mod\ } \pi L^++\,[L^+,L^+]),
$$
the elements of the form (B) can be also written in the form (3).
The elements in (A) with $i_k\le m$ account for the elements in
(4) and the elements in (A) with $i_k>m$ account for the elements
in (5). The later account for the terms $q_k(t)$ in $P(t)$. Thus
$Y$ generates $L^+(X)$ and so the canonical mapping of $L(Y)$, the
free Lie algebra over $\F_2[\pi]$ on the weighted set $Y$, into
$L^+$ is surjective. It is injective since $L(Y)$ and $L^+$ have
the same Poincar\'e series.
\end{proof}
\medskip

\begin{corollary}\label{Lq}
Let ${\tilde L}_{\rm mix}(X)=L_{\rm mix}(X)/\pi L_{\rm mix}(X)^+$
and let\, $Y$ be as in Theorem~\ref{L+}. Then ${\tilde L}_{\rm
mix} (X)^+=\bar L(Y)$, the free Lie algebra over $\F_2$ on $Y$.
Its enveloping algebra $\bar B$ is the subalgebra of $\bar A=\bar
A(X)$ (the free associative $\F_2$-algebra on $X$) generated by
$\tilde L(X)^+$. The $\bar B$-module $\bar A$ is free with basis
consisting of the elements $\xi_1^{i_1}\cdots\xi_m^{i_m}$
$(i_k=0,1)$.
\end{corollary}

This follows immediately from the fact that $A$ is a free
$B$-module with basis $\xi_1^{i_1}\cdots\xi_m^{i_m}$ $(i_k=0,1)$.

\section{Quadratic Lie Algebras}
If $\mathfrak g$ is a mixed Lie algebra we let $\tilde{\mathfrak
g}=\mathfrak g/\pi\mathfrak g^+$. Then $\tilde{\mathfrak g}$ is a
Lie algebra over $\F_2$ which we call the reduced algebra of
$\mathfrak g$. The operator $P$ on $\mathfrak g$ induces an
operator on $\tilde{\mathfrak g}$, also denoted by $P$, which is
zero in degree $>1$ and which, for homogeneous elements
$\xi,\eta$, satisfies

\begin{enumerate}[{\rm (QL1)}]
\item \ \ $P(\xi+\eta)=P(\xi)+P(\eta)+[\xi,\eta]$ if $\xi,\eta$ are of degree $1$,
\item \ \ $[P\xi,\eta]=[\xi,[\xi,\eta]]$ if $\xi$ is of degree
$1$.
\end{enumerate}
\medskip

Thus $\tilde{\mathfrak g}$ satisfies the axioms for a mixed Lie
algebra where $P(\xi)=0$ if $\xi$ is homogeneous of degree $>1$.
It is an example of what we call a quadratic Lie algebra.
\medskip

\begin{definition} [(Quadratic Lie Algebra)] A quadratic Lie algebra is a graded Lie algebra
$\mathfrak h=\oplus_{i\ge1}\mathfrak h_i$ over $\F_2$ together
with a mapping $P:\mathfrak h_1\rightarrow\mathfrak h_2$
satisfying {\rm(QL1)} and {\rm (QL2)}.
\end{definition}
\medskip

A homomorphism $f:\mathfrak h\rightarrow\mathfrak h'$ of quadratic
Lie algebras is a homomorphism of graded Lie algebras (over
$\F_2$) such that $f(P(s))=P(f(s))$ for every homogenous element
$s$ of degree $1$. By an ideal of $\mathfrak h$ we mean an ideal
$\mathfrak a$ of $\mathfrak h$ as a Lie algebra over $\F_2$ such
that $P(s)\in\mathfrak a$ for every element $s$ of $\mathfrak a$
of degree $1$. Every quadratic Lie algebra is a mixed Lie algebra
if we set $P\xi=0$ for every homogeneous element $\xi$ of degree
$1$. In this way Quadratic Lie algebras form a full subcategory of
the category of mixed Lie algebras.

If $A=\oplus_{i\ge0}A_i$ is a graded associative algebra over
$\F_2$ then the mapping $P: x\mapsto x^2$ of $A_1$ into $A_2$
together with the bracket $[x,y]=xy+yx$ defines the structure of a
quadratic Lie algebra on $A_+=\oplus_{i>0}A_i$. Indeed, we have
$(x+y)^2=x^2+y^2+xy+yx$ and
$$
[x,[x,y]]=[x,xy+yx]=x^2y+xyx+xyx+yx^2=[x^2,y].
$$
\medskip

\begin{definition}[(Derivation of a quadratic Lie algebra)]
If $\mathfrak h$ is a quadratic Lie algebra then by a derivation
of $\mathfrak h$ we mean an additive mapping $D:\mathfrak
h\rightarrow\mathfrak h$ that

\begin{enumerate}[{\rm (Der 1)}]
\item There is an integer $s\ge1$ such that $D(\mathfrak
h_n)\subseteq\mathfrak h_{n+s}$ $(s$ is the degree of $D)$,
\item\ \ $D(P(\xi))=[\xi,D(\xi)]$ if\, $\xi$ is homogeneous of degree
$1$,
\item\ \ $D[\xi,\eta]=[D(\xi),\eta]+[\xi,D(\eta)]$.
\end{enumerate}
\end{definition}

The set $\rm Der_{\rm quad}(\mathfrak h)$ of derivations of the
quadratic Lie algebra $\mathfrak h$ is a quadratic Lie algebra
under the operations of addition and Lie bracket
$[D_1,D_2]=D_1D_2+D_2D_1$ with $P(D)=D^2$ if $D$ is of degree $1$.
The grading is defined by the degree of a derivation.

If $\mathfrak a$ and $\mathfrak h$ are Lie algebras over $\F_2$
and $f$ is a homomorphism of $\mathfrak h$ into the Lie algebra of
derivations of $\mathfrak a$, the semi-direct product of
$\mathfrak a$ and $\mathfrak h$ is the direct product $\mathfrak
a\times\mathfrak h$ as vector spaces with the Lie algebra
structure given by
$$
[(\xi,\sigma),(\xi',\sigma')]=([\xi,\xi']+f(\sigma)(\xi')+f(\sigma')(\xi),[\sigma,\sigma']).
$$
We denote this Lie algebra by $\mathfrak a\times_f\mathfrak h$. We
will agree to identify $\mathfrak a$ and $\mathfrak h$ with their
canonical images in $\mathfrak a\times_f\mathfrak h$. If
$\mathfrak a$ and $\mathfrak h$ are graded then so is $\mathfrak
a\times_f\mathfrak h$ with $n$-th homogeneous component $\mathfrak
a_n\times\mathfrak h_n=\mathfrak a_n+\mathfrak h_n$.
\medskip

\begin{theorem} Let $\mathfrak a$ and $\mathfrak h$ be quadratic
Lie algebras and $f$ is a homomorphism of $\mathfrak h$ into $\rm
Der_{\rm quad}(\mathfrak a)$. If $(\xi,\sigma)$ is an element of
$\mathfrak a\times \mathfrak h$ of degree $1$ then
$$
P(\xi,\sigma)=(P(\xi)+f(\sigma)(\xi),P(\sigma))
$$
defines the structure of a quadratic Lie algebra on $\mathfrak
a\times_f\mathfrak h$.
\end{theorem}
\begin{proof}
Let $\xi+\sigma$, $\xi'+\sigma'$ be elements of $\mathfrak
a\times\mathfrak h$ of degree $1$. Then
\begin{align*}
&P(\xi+\sigma)+\xi'+\sigma')=P(\xi+\xi'+\sigma+\sigma')=\\
&P(\xi+\xi')+f(\sigma+\sigma')(\xi+\xi')+P(\sigma+\sigma')=\\
&P(\xi)+P(\xi')+[\xi,\xi']+f(\sigma)(\xi)+f(\sigma)(\xi')+f(\sigma')(\xi)+
f(\sigma')(\xi')+\\
&\ \ \ \ \ \ \ \ \ P(\sigma)+P(\sigma')+[\sigma,\sigma']=\\
&P(\xi+\sigma)+P(\xi'+\sigma')+[\xi+\sigma,\xi'+\sigma'].
\end{align*}
If $\xi+\sigma$ is of degree $1$ we have
\begin{align*}
&[P(\xi+\sigma),\xi'+\sigma']=[P(\xi)+f(\sigma)(\xi)+P(\sigma),\xi'+\sigma']=\\
&[P(\xi)+f(\sigma)(\xi),\xi']+f(P(\sigma))(\xi')+f(\sigma')((P(\xi)+f(\sigma)(\xi)+[P(\sigma),\sigma']=\\
&[P(\xi),\xi']+[f(\sigma)\xi,\xi']+f(\sigma)^2\xi'+[\xi,f(\sigma')(\xi)]+f(\sigma')f(\sigma)(\xi)+[P(\sigma),\sigma']=\\
&[\xi,[\xi,\xi']]+[f(\sigma)(\xi,\xi']+f(\sigma)^2(\xi')+
[\xi,f(\sigma')(\xi)+f(\sigma')f(\sigma)(\xi)+[\sigma,[\sigma,\sigma']]=\\
&[\xi,[\xi,\xi']]+[\xi,f(\sigma)(\xi')]+[\xi,f(\sigma')(\xi)]+f(\sigma)([\xi,\xi']+
f(\sigma)^2(\xi')+f(\sigma)f(\sigma')(\xi)\\
&\quad +f([\sigma,\sigma'])(\xi)+[\sigma,[\sigma,\sigma]]=\\
&[\xi+\sigma,[\xi,\xi']+f(\sigma)(\xi')+f(\sigma')(\xi)+[\sigma,\sigma']]=[\xi+\sigma,[\xi+\sigma,\xi'+\sigma']].
\end{align*}
\end{proof}

If $X$ is a homogeneous subset of the quadratic Lie algebra
$\mathfrak h$ then the quadratic subalgebra of $\mathfrak h$
generated by $X$ is the smallest Lie subalgebra $\mathfrak a$ of
$\mathfrak h$ which contains $X$ and which contains $P(x)$ for
every $x\in X$ of degree $1$. Let $\mathfrak h^*=P(\mathfrak
h_1)+[\mathfrak h,\mathfrak h]$. Then $\mathfrak h^*$ is a vector
subspace of $\mathfrak h$ by (QL1). The proof of the following
result is left to the reader.
\medskip

\begin{proposition}
The subset $X$ generates the quadratic Lie algebra $\mathfrak h$
if and only its image in the vector space $\mathfrak h/\mathfrak
h^*$ is a generating set.
\end{proposition}

If $X$ is a weighted set then the natural map of $\tilde L_{\rm
mix} (X)=L_{\rm mix}(X)/\pi L_{\rm mix}(X)^+$ into $\bar
A(X)=A(X)/\pi A(X)$ is injective map of quadratic Lie algebras. We
use this to identify $\tilde L_{\rm mix}(X)$ with the quadratic
subalgebra of the free associative algebra $\bar A(X)$ over $\F_2$
generated by $X$. If $\bar L(X)$ is the Lie subalgebra of $\bar
A(X)$ generated by $X$ we have
$$
\tilde L_{\rm mix}(X)=\bar L(X)+\sum_{s\in S}\F_2s^2,
$$
where $S$ is the set of elements of $X$ of degree $1$ and
$P(s)=s^2$ for $s\in S$. The Lie algebra $\bar L(X)$ is the free
Lie algebra over $\F_2$ on $X$. Note that $\tilde L_{\rm mix}
(X)/\tilde L_{\rm mix}(X)^*=\bar L(X)/[\bar L(X),\bar L(X)]$.
\medskip

\begin{proposition}
The Lie algebra $\tilde L_{\rm mix}(X)$ is the free quadratic Lie
algebra on the set $X$.
\end{proposition}

\begin{proof}
Let $f$ be a weight preserving map of $X$ into a quadratic Lie
algebra $\mathfrak h$. Then $f$ extends uniquely to a Lie algebra
homomorphism $\varphi_0$ of $\bar L(X)$ into $\mathfrak h$. The
only way to extend $\varphi_0$ to a quadratic Lie algebra
homomorphism $\varphi$ of $\tilde L_{\rm mix}(X)$ into $\mathfrak
h$ is to define $\varphi(P(s))=P(\varphi(s))$ for any $s\in S$ and
to extend by linearity to all of $\tilde L_{\rm mix}(X)$. A
straightforward verification yields that
$\varphi([P(s),y])=[\varphi(P(s)),\varphi(y)]$ for any $y\in\bar
L(X)$ and that $\varphi([P(s),P(t)]=[\varphi(P(s)),\varphi(P(t))]$
for any $s,t\in S$ and hence that $\varphi$ is a homomorphism of
quadratic Lie algebras.
\end{proof}
\medskip

Every quadratic Lie algebra $\mathfrak h$ has a universal
enveloping algebra $U=U_{\rm quad}(\mathfrak h)$. More precisely,
there is a graded associative algebra $U$ over $\F_2$ and a
quadratic Lie algebra homomorphism $f$ of $\mathfrak h$ into $U_+$
such that for every quadratic Lie algebra homomorphism $\varphi_0$
of $\mathfrak h$ into an associative algebra $B$ over $\F_2$ there
is a unique algebra homomorphism $\varphi$ of $U$ into $B$
satisfying $\varphi\circ f=\varphi_0$. We have $U_{\rm
quad}({\tilde L}_{\rm mix}(X))={\bar A}(X)$ since ${\bar A}(X)$
has the correct universal property. More generally, we have
\medskip

\begin{proposition}
Let $\mathfrak g={\tilde L}_{\rm mix}(X)/\mathfrak r$ be a
presentation of a quadratic Lie algebra $\mathfrak g$ and let
$\mathfrak R$ be the ideal of ${\bar A}(X)=U_{quad}({\tilde
L}_{\rm mix}(X))$ generated by the image of $\mathfrak r$. Then
$$
{\bar A}(X)/\mathfrak R=U_{\rm quad}(\mathfrak g).
$$
\end{proposition}

\begin{proposition} Let $\mathfrak g$ be a mixed Lie algebra and
$\tilde{\mathfrak g}=\mathfrak g/\pi{\mathfrak g}^+$ the reduced
algebra of $\mathfrak g$. If $U=U_{\rm mix}(\mathfrak g)$ then
$U_{\rm quad}(\tilde{\mathfrak g})=U/\pi U$.
\end{proposition}
\begin{proof}
If $\mathfrak g=L_{\rm mix}(X)/\mathfrak r$ then $\tilde{\mathfrak
g}={\tilde L}_{\rm mix}(X)/\tilde{\mathfrak r}$, where
$\tilde{\mathfrak r}$ is the image of $\mathfrak r$ in ${\tilde
L}_{\rm mix}(X)$. Then
$$U_{\rm quad}(\tilde{\mathfrak g})={\bar
A}(X)/\tilde{\mathfrak R},
$$
where $\tilde{\mathfrak R}$ is the image of $\mathfrak R$ in
${\bar A}(X)$.
\end{proof}

\section{Strongly Free Sequences}
Let $\rho_1,\ldots,\rho_m\in L=L_{\rm mix}(X)$ with $\rho_i$
homogeneous of degree $h_i>1$ and let $\mathfrak r$ be the ideal
of the free mixed Lie algebra $L$ generated by
$\rho_1,\ldots,\rho_m$. Let $\mathfrak g= L/\mathfrak r$. Then
$M=\mathfrak r/[\mathfrak r,\mathfrak r]$ is a module over the
enveloping algebra $U=U_{\rm mix}({\mathfrak g})$ via the adjoint
representation.

\begin{definition}
The sequence $\rho_1,\ldots,\rho_m$ is said to be strongly free in
$L$ if the following conditions hold.
\begin{enumerate}[\rm(i)]
\item The $\F_2[\pi]$-module $U$ is torsion free.
\item The $U$-module $M$ is free on the images of
$\rho_1,\ldots,\rho_m$.
\end{enumerate}
\end{definition}
\medskip

Let $\tilde\rho_i$ be the image of $\rho_i$ in $\tilde L=\tilde
L_{\rm mix}(X)$ and let $\tilde{\mathfrak r}$ be the ideal of
$\tilde L$ generated by $\tilde\rho_1,\ldots,\tilde\rho_m$. Let
$\tilde{\mathfrak g}= \tilde L/\tilde{\mathfrak r}$. Then $\tilde
M= \tilde{\mathfrak r}/[\tilde{\mathfrak r},\tilde{\mathfrak r}]$
is a module over the enveloping algebra $\tilde U=U_{\rm
quad}({\tilde{\mathfrak g}})$ via the adjoint representation.

\begin{definition}
The sequence $\tilde\rho_1,\ldots,\tilde\rho_m$ is said to be a
strongly free in\, $\tilde L$ if the $\tilde U$-module $\tilde M$
is free on the images of $\tilde\rho_1,\ldots,\tilde\rho_m$.
\end{definition}
\medskip

Let $X=\{\xi_1,\ldots,\xi_d\}$ with $\xi_i$ of weight $e_i$.
\medskip

\begin{theorem}\label{P}
The sequence $\tilde\rho_1,\ldots,\tilde\rho_m$ is strongly free
in $\tilde L$ if and only if the Poincar\'e series of\, $\tilde U$
is
$$
1/(1-(t^{e_1}+\cdots+t^{e_d})+t^{h_1}+\cdots+t^{h_m}).
$$
\end{theorem}
\begin{proof}
Let $\mathfrak R$ be the ideal of $\bar A(X)$ generated by
$\tilde{\mathfrak r}$. Then ${\bar A}(X)/\mathfrak R=U_{\rm
quad}({\tilde{\mathfrak g}})=\tilde U$. If $I$ is the augmentation
ideal of $V={\bar A}(X)$ and $J$ is the augmentation ideal of
$W=U_{\rm quad}(\tilde{\mathfrak r})$ then, by tensoring the exact
sequence $0\rightarrow I\rightarrow V\rightarrow \F_2\rightarrow
0$ with $\F_2=W/J$ over $W$, we obtain the exact sequence
$$
{\rm Tor}_1^W(\F_2,V)\rightarrow \tilde{\mathfrak
r}/[\tilde{\mathfrak r},\tilde{\mathfrak r}]\rightarrow
I/{\mathfrak R}I\rightarrow V/\mathfrak R\rightarrow
\F_2\rightarrow 0
$$
using the fact that
\begin{enumerate}
\item If $M$ is a $W$-module then $M\otimes_W(W/J)=M/JM$;
\item $\mathfrak R={\tilde{\mathfrak r}}V=V{\tilde{\mathfrak r}}$;
\item ${\rm Tor}_1^W(\F_2,\F_2)=\tilde{\mathfrak r}/[\tilde{\mathfrak r},\tilde{\mathfrak r}]$ (cf.
\cite{CE}, Ch. XIII, \S 2).
\end{enumerate}

The map $\tilde{\mathfrak r}/[\tilde{\mathfrak r},\tilde{\mathfrak
r}]\rightarrow I/{\mathfrak R}I$ is induced by the inclusion
$\tilde{\mathfrak r}\subseteq I$. Since $I$ is the direct sum of
the left ideals $V\xi_i$. The $\tilde U$-module $I/\mathfrak R$ is
the direct sum of the free $\tilde U$-submodules $Ug_i$ where
$g_i$ is the image of $\xi_i$ in $U={\bar A}(X)/\mathfrak R$.
Since $\tilde{\mathfrak r}\subset \tilde L$ the algebra $V={\bar
A}(X)$ is a free $W$-module by Corollary~\ref{Lq} and the
Birkhoff-Witt Theorem for Lie algebras over $\F_2$. In this case
we have the exact sequence
$$
0\rightarrow\tilde{\mathfrak r}/[\tilde{\mathfrak
r},\tilde{\mathfrak r}]\rightarrow I/{\mathfrak R}I\rightarrow
{\bar A}(X)/\mathfrak R\rightarrow \F_2\rightarrow 0.
$$
Expressing $\tilde M=\tilde{\mathfrak r}/[\tilde{\mathfrak
r},\tilde{\mathfrak r}]$ as a quotient $\tilde U^m/N$ using the
relators $\tilde\rho_i$, we obtain the exact sequence of graded
modules whose homogeneous components are finitely generated free
$\F_2$-modules
$$
0\rightarrow N\rightarrow \oplus_{j=1}^m\tilde
U[h_j]\rightarrow\oplus_{j=1}^d\tilde U[e_j]\rightarrow
U\rightarrow \F_2\rightarrow 0
$$
where $\tilde U[n]=\tilde U$ but with degrees shifted by $n$; by
definition, $\tilde U[n](t)=t^n\tilde U(t)$. We have $N=0$ if and
only if $\tilde M$ is a free $\tilde U$-module on the images of
the $\tilde\rho_i$.

Taking Poincar\'e series in the above long exact sequence, we get
$$
N(t)-(t^{h_1}+\cdots+t^{h_m})\tilde
U(t)+(t^{e_1}+\cdots+t^{e_d})\tilde U(t)-\tilde U(t)+1=0.
$$
Solving for $\tilde U(t)$, we get $\tilde U(t)=P(t)+N(t)P(t)$,
where
$$
P(t)=\frac{1}{1-(t^{e_1}+\cdots+t^{e_d})+t^{h_1}+\cdots+t^{h_m}}.
$$
Hence $N(t)=0\iff\tilde U(t)=P(t)$.
\end{proof}
\medskip

\begin{theorem}\label{sfred}
The sequence $\rho_1,\ldots,\rho_m$ is strongly free in $L=L_{\rm
mix} (X)$ if and only if the sequence
$\tilde\rho_1,\ldots,\tilde\rho_m$ is strongly free in $\tilde L$.
\end{theorem}
\begin{proof}
If $\rho_1,\ldots,\rho_m$ is a strongly free sequence then the
enveloping algebra $U$ of the mixed Lie algebra $\mathfrak
g=L/\mathfrak r$ is a torsion free $\F_2[\pi]$-module. By the
Birkhoff-Witt Theorem for mixed Lie algebras, the canonical
mapping of $\mathfrak g$ into $U$ is injective. Hence $\mathfrak
g^+=L^+/\mathfrak r$ is a torsion free $\F_2[\pi]$-module. If $B$
is the subalgebra of $A=A(X)$ generated by $L^+$ then $B$ is the
enveloping algebra of $L^+$. By Birkhoff-Witt the canonical
mapping of the enveloping algebra $W$ of $\mathfrak r$ into $B$ is
injective and $B$ is a free $W$-module. Since $A$ is a free
$B$-module it follows that $A$ is a free $W$-module. Thus, if
$M=\mathfrak r/[\mathfrak r,\mathfrak r]$ and $\mathfrak R$ the
ideal of $A$ generated by $\mathfrak r$ and $I$ the augmentation
ideal of $A$, we have an exact sequence
$$
0\rightarrow M\rightarrow I/{\mathfrak R}I\rightarrow A/\mathfrak
R\rightarrow \F_2[\pi]\rightarrow 0.
$$
As in the proof of Theorem~\ref{P} we obtain that the Poincar\'e
series of $U$ is
$$
Q(t)=\frac{1}{(1-t)(1-(t^{e_1}+\cdots+t^{e_d})+t^{h_1}+\cdots+t^{h_m})}.
$$
If $\tilde U$ is the enveloping algebra of $\tilde
L/(\tilde\rho_1,\ldots,\tilde\rho_m)$ we have $\tilde U=U/\pi
U=U\otimes_{\F_2}\F_2[\pi]$. Since $U$ is torsion free over
$\F_2[\pi]$ the Poincar\'e series of $\bar U$ is $(1-t)Q(t)$ which
proves that the sequence $\tilde\rho_1,\ldots,\tilde\rho_m$ is
strongly free.

Conversely, suppose that the sequence
$\tilde\rho_1,\ldots,\tilde\rho_m$ is strongly free in $\tilde L$.
We have the exact sequence of graded vector spaces over $\F_2$
$$
0\rightarrow K\rightarrow M \rightarrow U[e_1]\oplus\cdots\oplus
U[e_d]\rightarrow U\rightarrow \F_2\rightarrow 0.
$$
Taking Poincar\'e series we get
$$
K(t)-M(t)+(t^{e_1}+\cdots+t^{e_d})U(t)-U(t)+\frac{1}{1-t}=0
$$
from which we get
$M(t)=K(t)-(1-(t^{e_1}+\cdots+t^{e_d}))U(t)+1/(1-t)$. Hence
$$
\frac{M(t)}{1-(t^{e_1}+\cdots+t^{e_d})}=\frac{K(t)}{1-(t^{e_1}+\cdots+t^{e_d})}
+\frac{1}{(1-t)(1-(t^{e_1}+\cdots+t^{e_d}))}-U(t).
$$
Now suppose that $\tilde{\rho}_1,\ldots,\tilde{\rho}_m$ is
strongly free. Then, if $\tilde{\mathfrak r}$ is the ideal of
$\tilde{L}$ generated by $\tilde{\rho}_1,\ldots,\tilde{\rho}_m$
and $\tilde M=\tilde{\mathfrak r}/[\tilde{\mathfrak
r},\tilde{\mathfrak r}]$, we have surjections
$$
\tilde{U}[h_1]\oplus\cdots\oplus\tilde{U}[h_m]\rightarrow
\tilde{M}\rightarrow \tilde{\mathfrak r}/[\tilde{\mathfrak
r},\tilde{\mathfrak r}]
$$
whose composite is an isomorphism. It follows that
$$\tilde{M}\cong\tilde{\mathfrak r}/[\tilde{\mathfrak
r},\tilde{\mathfrak r}]\cong
\tilde{U}[h_1]\oplus\cdots\oplus\tilde{U}[h_m],
$$
$$
M(t)\le\frac{\tilde{M}(t)}{1-t}=
\frac{1}{1-t}\cdot\frac{t^{h_1}+\cdots+t^{h_m}}{1-(t^{e_1}+\cdots+t^{e_d})+t^{h_1}+\cdots+t^{h_m}},
$$
$$
U(t)\le
\frac{\tilde{U}(t)}{1-t}=\frac{1}{1-t}\cdot\frac{1}{1-(t^{e_1}+\cdots+t^{e_d})+t^{h_1}+\cdots+t^{h_m}}\cdot
$$
Using the fact that $K(t)\ge0$, we get
\begin{align*}
&\frac{M(t)}{1-(t^{e_1}+\cdots+t^{e_d})}\ge
\frac{1}{(1-t)(1-(t^{e_1}+\cdots+t^{e_d}))}-\frac{\tilde{U}(t)}{1-t}=\\
&\frac{1}{1-t}(\frac{1}{(1-(t^{e_1}+\cdots+t^{e_d})}-\frac{1}{1-(t^{e_1}+\cdots+t^{e_md})+t^{h_1}+\cdots+t^{h_m}})=\\
&\frac{\tilde{M}(t)}{(1-t)(1-(t^{e_1}+\cdots+t^{e_d}))}\ge\frac{M(t)}{1-(t^{e_1}+\cdots+t^{e_d})}.
\end{align*}
It follows that $K(t)=0$, $U(t)=\tilde{U}(t)/(1-t)$ and
$M(t)=\tilde{M}/(1-t)$. Hence $U$ is a free $\F_2[\pi]$-module and
$M$ is a free $U$-module since we have a natural surjection
$$
U[h_1]\oplus\cdots U[h_m]\rightarrow M
$$
with both sides having the same Poincar\'e series.
\end{proof}
\medskip

In general it is very difficult to determine whether a sequence in
$\tilde L$ is strongly free but we can construct a large supply
using the following elimination theorem for free quadratic Lie
algebras.
\medskip

\begin{theorem}[(Elimination Theorem)]
Let $S$ be a subset of the weighted set $X$ and let $\mathfrak a$
be the ideal of the free quadratic Lie algebra\, $\tilde L_{\rm
mix} (X)$ generated by $X-S$. Then $\mathfrak a$ is a free
quadratic Lie algebra with basis
$$
\ad(\sigma_1)\ad(\sigma_2)\cdots\ad(\sigma_n)(\xi),\quad (n\ge0,\
\sigma_i\in S,\ \xi\in X-S).
$$
\end{theorem}
\begin{proof}
We first show that the quadratic Lie algebra $\tilde L_{\rm
mix}(X)$ is the semi-direct product of the quadratic Lie algebras
$\mathfrak a$ and ${\tilde L}_{\rm mix}(S)$. Let $f$ be the
adjoint representation of $\tilde L_{\rm mix}(S)$ on $\mathfrak
a$. Then $f$ is a homomorphism of the quadratic Lie algebra
${\tilde L}_{\rm mix}(S)$ into the quadratic Lie algebra $\rm
Der_{\rm quad}(\mathfrak a$) of derivations of the quadratic Lie
algebra $\mathfrak a$. More precisely, if $f(\sigma)=D$ then $
f(P(\sigma))=D^2$ and $D(P(\xi))= [\xi,D(\xi)]$ if $\sigma,\xi$
are homogeneous of degree $1$. Every element of ${\tilde L}_{\rm
mix} (X)$ can be uniquely written in the form $\xi+\sigma$ with
$\xi\in\mathfrak a,\sigma\in{\tilde L}_{\rm mix}(S)$. We have
$$
[\xi_1+\sigma_1,\xi_2+\sigma_2]=[\xi_1,\xi_2]+f(\sigma_1)(\xi_2)+f(\sigma_2)(\xi_1)+[\sigma_1,\sigma_2]
$$
and $P(\xi+\sigma)=P(\xi)+f(\sigma)(\xi)+P(\sigma)$ if
$\xi,\sigma$ are of degree $1$. As a quadratic Lie algebra,
$\mathfrak a$ is generated by the family of elements
$$
\ad(\sigma_1)\ad(\sigma_2)\cdots\ad(\sigma_n)(\xi),
\quad(n\ge0,\sigma_i\in S,\xi\in X-S).
$$
If $\sigma\in S$ and $f(\sigma)=D$ then
$$
D(\ad(\sigma_1)\ad(\sigma_2)\cdots\ad(\sigma_n)(\xi))=\ad(\sigma)\ad(\sigma_1)\ad(\sigma_2)\cdots\ad(\sigma_n)(\xi).
$$

Let $T$ be the family of elements
$(\sigma_1,\sigma_2,\ldots,\sigma_n,\xi)$ with $n\ge0,\sigma_i\in
S,\xi\in X-S$ and weight equal to the sum of the weights of the
components $\sigma_i,\xi$. Let $\varphi_1$ be the quadratic Lie
algebra homomorphism of ${\tilde L}_{\rm mix}(T)$ into $\mathfrak
a$ such that
$$
\varphi_1(\sigma_1,\sigma_2,\ldots,\sigma_n,\xi)=\ad(\sigma_1)\ad(\sigma_2)\cdots\ad(\sigma_n)(\xi).
$$
Since $\varphi_1$ is surjective it suffices to prove $\varphi_1$
is injective. Let $g$ be the quadratic Lie algebra homomorphism of
${\tilde L}_{\rm mix}(S)$ into ${\rm Der}_{\rm quad}({\tilde
L}_{\rm mix}(T))$ where, for $\sigma\in S$, we define $g(\sigma)$
be the derivation which takes
$(\sigma_1,\sigma_2,\ldots,\sigma_n,\xi)$ into
$(\sigma,\sigma_1,\sigma_2,\ldots,\sigma_n,\xi)$. That such a
derivation exists follows from the fact that the derivations $D$
of the free Lie algebra $\bar L(T)$ can be assigned arbitrarily
and can be uniquely extended to derivations of the quadratic Lie
algebra $\tilde L(T)$ by defining $D(\xi^2)=[\xi,D(\xi)]$ if $\xi$
is an element of $T$ of degree $1$. Let $L$ be the semi-direct
product of ${\tilde L}_{\rm mix}(T)$ and ${\tilde L}_{\rm mix}(S)$
with respect to the homomorphism $g$. Every element of $L$ can be
uniquely written in the form $\xi+\sigma$ with $\xi\in{\tilde
L}_{\rm mix}(T),\sigma\in{\tilde L}_{\rm mix}(S)$. Then
$$
[\xi_1+\sigma_1,\xi_2+\sigma_2]=[\xi_1,\xi_2]+g(\sigma_1)(\xi_2)+g(\sigma_2)(\xi_1)+[\sigma_1,\sigma_2].
$$
and $P(\xi+\sigma)=P(\xi)+g(\sigma)(\xi)+P(\sigma)$ if
$\xi,\sigma$ are of degree $1$. Since
$\varphi_1(g(\sigma)(\xi))=f(\sigma)(\varphi_1(\xi))$ we see that
there is a unique homomorphism $\varphi$ of $L$ into ${\tilde
L}_{\rm mix}(X)$ which restricts to $\varphi_1$ and is the
identity on ${\tilde L}_{\rm mix}(S)$. If $\psi$ is the
homomorphism of $\tilde L(X)$ into $L$ which is the identity on
$X$ we have $\varphi\circ \psi$ and $\psi\circ\varphi$ identity
maps so that $\varphi$ and hence $\varphi_1$ is bijective.
\end{proof}
\medskip

\begin{corollary}
If $B$ is the enveloping algebra of ${\tilde L}_{\rm
mix}(S)={\tilde L}_{\rm mix}(X)/\mathfrak a$ then, via the adjoint
representation, $\mathfrak a/[\mathfrak a,\mathfrak a]$ is a free
$B$-module with basis the images of the elements $\xi\in X-S$.
\end{corollary}
\medskip

Let $X$ be a finite weighted set and let $S\subset X$. Let
$\mathfrak a$ be the ideal of $\tilde L=\tilde L_{\rm mix}(X)$
generated by $X-S$ and let $B$ be the enveloping algebra of
$\tilde L/\mathfrak a$.

\begin{theorem}\label{SF}
Let $T=\{\tau_1,\ldots,\tau_t\}\subset \mathfrak a$ whose elements
are homogeneous of degree $\ge2$ and $B$-independent modulo
$\mathfrak a^*$. If $\rho_1,\ldots,\rho_m$ are homogeneous
elements of $\mathfrak a$ which lie in the $\F_2$-span of\, $T$
modulo $\mathfrak a^*$ and which are linearly independent over
$\F_2$ modulo $\mathfrak a^*$ then the sequence
$\rho_1,\ldots,\rho_m$ is strongly free in $\tilde L$.
\end{theorem}
\begin{proof}
Let $\mathfrak r$ is the ideal of $\tilde L$ generated by
$\rho_1,\ldots,\rho_m$ and let $U=U_{\rm quad}$ be the enveloping
algebra of $\tilde L/\mathfrak r$. The elements
$$
\ad(\sigma_1)\ad(\sigma_2)\cdots \ad(\sigma_n)(\rho_j)
$$
with $1\le j\le m$, $n\ge0$, $\sigma_i\in S$ generate $\mathfrak
r$ as an ideal of the quadratic Lie algebra $\mathfrak a$. Suppose
that these elements form part of a basis of the free quadratic Lie
algebra $\mathfrak a$. The elimination theorem then shows that
$M={\mathfrak r}/[{\mathfrak r},{\mathfrak r}]$ is a free module
over the enveloping algebra $C$ of ${\mathfrak a}/{\mathfrak r}$
with the images of these elements as basis. Now let $\mu_i$ be the
image of $\rho_i$ in $M$ and suppose that $\sum_iu_i\mu_i=0$ with
$u_i\in U$. Then, since every $u_i$ can be written in the form
$$
u_i=\sum \bar{c}_{ij}w_j
$$
where the $w_j$ are distinct products of elements of $S$ and
$c_{ij}\in C$ with $\bar{c}_{ij}$ its image in $U$, the dependence
relation
$$
0=\sum_iu_i\mu_i=\sum_{i,j}(\bar{c}_{ij}\,w_j)\mu_i=\sum_{i,j}c_{ij}(w_j\mu_i)
$$
implies that all $c_{ij}$ are zero and hence that each $u_i$ is
zero.

To show that the elements of the form
$\ad(\sigma_1)\ad(\sigma_2)\cdots \ad(\sigma_n)(\rho_j)$ are part
of a Lie algebra basis of $\mathfrak a$ it suffices to show that
$\rho_1,\ldots,\rho_m$ are $B$-independent modulo $\mathfrak a^*$.
We now work modulo $\mathfrak a^*$. If $H$ is the $\F_2$-span of
$\rho_1,\ldots,\rho_m$, we can find a basis
$\gamma_1,\ldots,\gamma_m$ of $H$ such that
$$
\gamma_i=a_i\alpha_i+\sum_{j=1}^sa_{ij}\beta_j
$$
where $a_i,a_{ij}\in \F_2$, $a_i\ne0$, $m+s=t$,
$T=\{\alpha_1,\ldots,\alpha_m,\beta_1,\ldots,\beta_s\}$. If
$u_1,\ldots,u_m\in B$, we have
$$
\sum_{i=1}^m u_i\gamma_i=
\sum_{i=1}^ma_iu_i\alpha_i+\sum_{j=1}^s(\sum_{i=1}^ma_{ij}u_i)\beta_j.
$$
If\ $\sum_{i=1}^m u_i\gamma_i=0 \mod \mathfrak a^*$ then by the
$B$-independence of the elements of $T$ we $a_iu_i=0$ so that
$u_i=0$ for all $i$ which implies the $B$-independence of
$\gamma_1,\ldots,\gamma_m$ and hence of $\rho_1,\ldots,\rho_m$.
\end{proof}
\medskip

\begin{corollary}\label{sf}
Let $X=\{\xi_1,\ldots,\xi_d\}$ with $d\ge 4$ even and let
$\rho_1,\ldots,\rho_d\in \tilde L_{\rm mix}(X)$ with
$$
\rho_i=a_i\xi_i^2+\sum_{j=1}^d \ell_{ij}[\xi_i,\xi_j],
$$
where {\rm (a)} $a_i=0$ if $i$ is odd, {\rm (b)} $\ell_{ij}=0$ if
$i,j$ odd, {\rm (c)}
$\ell_{12}=\ell_{23}=\ldots=\ell_{d-1,d}=\ell_{d1}=1$ and {\rm
(d)} $\ell_{1d}\ell_{d,d-1}\cdots\ell_{32}\ell_{21}=0$. Then the
sequence $\rho_1,\ldots,\rho_d$ is strongly free.
\end{corollary}
\begin{proof}
Let $\mathfrak a$ be the ideal of ${\tilde L}_{\rm mix}(X)$
generated by the $\xi_i$ with $i$ even and let $\mathfrak b$ be
the subspace of $\mathfrak a_2$ generated by the $\xi_i^2$,
$[\xi_i,\xi_j]$ with $i,j$ even. Then the $\rho_i$ are in
$\mathfrak a$ and their images in $V=(\mathfrak a/\mathfrak
a^*)_2=\mathfrak a_2/\mathfrak b$ are linearly independent.
Indeed, the images in $V$ of the elements $[\xi_i,\xi_j]$ with $i$
odd, $j$ even $i<j$ form a basis for $V$ which we order
lexicographically. If $A$ is the matrix representation of
$\rho_1,\ldots,\rho_d$ with respect to this basis, the $d$ columns
$(1,2),(2,3),(3,4),\ldots,(1,d)$ of $A$ form the matrix
$$
\begin{bmatrix} \ell_{12}&0&0&\cdots&0&-\ell_{1m}\\
                \ell_{21}&\ell_{23}&0&\cdots&0&0\\
                0&\ell_{32}&\ell_{34}&\cdots&0&0\\
                0&0&\ell_{43}&\cdots&0&0\\
                \vdots&\vdots&\vdots&&\vdots&\vdots\\
                0&0&0&\cdots&\ell_{m,m-1}&0\\
                0&0&0&\cdots&\ell_{m,m-1}&\ell_{m1}\end{bmatrix}
$$
which has determinant
$\ell_{12}\ell_{23}\cdots\ell_{m-1,m}\ell_{m1}+\ell_{1m}\ell_{21}\ell_{32}\cdots\ell_{m,m-1}=1$.
\end{proof}
\medskip

\noindent{\bf Example 4.9.} If $d\ge4$ is even then
$$
a_1\xi_1^2+[\xi_1,\xi_2],a_2\xi_2^2+[\xi_2,\xi_3],\ldots,a_d\xi_d^2+[\xi_d,\xi_1]
$$
is a strongly free sequence if $a_i=0$ for $i$ odd.
\medskip

\section{Mild Groups}
Let $F=F(x_1,\ldots,x_d)$ be the free pro-$2$-group on
$x_1,\ldots,x_d$ and let $G=F/R$ with $R$ the closed normal
subgroup of $F$ generated by $r_1,\ldots,r_m$. Let $(F_n)$ be the
filtration of $F$ induced by the $(x,\tau)$-filtration of $F$. It
is induced by the $(x,\tau)$-filtration of $\Lambda=\Z_2[[F]]$.
Let $G_n$ be the image of $F_n$ in $G$ and let $\Gamma_n$ be the
image of $\Lambda_n$ in $\Gamma=\Z_2[[G]]$.

Let $\rho_i$ be the initial form of $r_i$ with respect to the
$(x,\tau)$-filtration of $F$ ; by definition, if $r\in F_k,\
r\notin F_{k+1}$, the initial form of $r$ is the image of $r$ in
$L_k(F)=\gr_k(F)$. We assume that the degree $h_i$ of $\rho_i$ is
$>1$.
\medskip

\begin{definition}[(Strongly Free Presentation)]
The presentation $G=F/R$ is strongly free if
$\rho_1,\ldots,\rho_m$ is strongly free in $L_{\rm mix}(F)$.
\end{definition}
\medskip

\begin{definition}[(Mild Group)]
A pro-$2$-group $G$ is said to be {\bf weakly mild} if it has a
minimal presentation $G=F/R$ of finite type which is strongly free
with respect to some $(x,\tau)$-filtration of $F$. It is called
{\bf mild} if the $\tau_i=1$ for all $i$ in which case the
$(x,\tau)$-filtration is the lower $2$-central series of $F$.
\end{definition}
\medskip

\begin{theorem}\label{Cent}
Let $F/R$ be a strongly free presentation of $G$ with
$R=(r_1,\ldots,r_m)$. Let $\mathfrak r$ is the ideal of $L(F(X))$
generated by the initial forms $\rho_1,\ldots,\rho_m$ of the
defining relators $r_1,\ldots,r_m$. Then
\begin{enumerate}[\rm (a)]
\item $L(G)=L(F)/\mathfrak r$.
\item The group $R/[R,R]$ is a free $\Z_2[[G]]$-module on the images of
$r_1,\ldots,r_m$.
\item The presentation $G=F/R$ is minimal and $\cd(G)=2$.
\item The enveloping algebra of $L(G)$ is the graded algebra
associated to the filtration $(\Gamma_n)$ of $\Gamma=\Z_2[[G]]$,
where $\Gamma_n$ is the image of $\Lambda_n$ in $G$.
\item The filtration $(G_n)$ of $G$ is induced by the filtration
$(\Gamma_n)$ of $\Gamma$.
\item The Poincar\'e series of $\gr(\Gamma)$ is
$1/(1-t)(1-(t^{\tau_1}+\cdots+t^{\tau_d})+t^{h_1}+\ldots+t^{h_m}))$.
\item If $b_n=\dim {\tilde L}_n$ then
the Poincar\'e series of $\gr(\Gamma)/\pi\,\gr(\Gamma)$ is equal
to
$$
(1+t)^r\prod_{n\ge2}(1-t^n)^{-b_n},
$$
where $r=b_1$ is the number of $i$ with $\tau_i=1$.
\item If $b_n$, $r$ are as in (g) and
$$
1-(t^{\tau_1}+\cdots+t^{\tau_d})+t^{h_1}+\ldots+t^{h_m})=(1-\alpha_1t)\cdots(1-\alpha_st)
$$
then $a_n=\sum_{k=2}^nb_k$ with
$$
b_n=\frac{1}{n}\sum_{\ell|n}\mu(\frac{n}{\ell})(\alpha_1^\ell+\cdots+\alpha_s^\ell+(-1)^\ell
r).
$$
\end{enumerate}
\end{theorem}
\medskip

Except for (g) and (h), the proof this theorem is the same as the
proof of Theorem 4.1 in \cite{La} except that the freeness of the
Lie algebra $\mathfrak r$ over $\F_2[\pi]$ is deduced from the
fact that $\mathfrak r$ is an ideal of the free Lie algebra
$L_{\rm mix}(X)^+$ and that $L_{\rm mix}(X)^+/\mathfrak r$ a
torsion free $\F_2[\pi]$-module.

To prove (g) and (h) let $A$ be the enveloping algebra of the
mixed Lie algebra $L=L_{\rm mix}(F(X))$ and let $B$ be the
enveloping algebra of the $\F_2[\pi]$-Lie algebra $L^+$. Then $L$
is the free mixed Lie algebra on $\xi_1,\ldots,\xi_d$, where
$\xi_i$ is the image of $x_i$ in $\gr_{\tau_i}(F)$. By
Theorem~\ref{L+}, $L^+$ is a free Lie algebra over $\F_2[\pi]$ and
the canonical map of $B$ into $A$ is injective. Moreover, assuming
that $\xi_1,\ldots,\xi_s$ the $\xi_i$ of degree $1$, then $A$ is a
free $B$-module with basis $\xi_1^{e_1}\cdots\xi_s^{e_s}$
($e_i=0,1$). If ${\mathfrak r}_B$ be the ideal of $B$ generated by
$\mathfrak r$ then
$$
{\mathfrak
r}_A=\sum_{e_i=0,1}\xi_1^{e_1}\cdots\xi_s^{e_s}{\mathfrak r}_B
$$
is the ideal of $A$ generated by $\mathfrak r$. It follows that
the canonical map of $B/{\mathfrak r}_B$ into $A/{\mathfrak r}_A$
is injective and that $A/{\mathfrak r}_A$ is a free $B/{\mathfrak
r}_B$-module with basis $\xi_1^{e_1}\cdots\xi_s^{e_s}$
($e_i=0,1$). The algebra $U=A/{\mathfrak r}_A$ is the enveloping
algebra of the mixed Lie algebra $\mathfrak g=L/\mathfrak r$ and
$V=B/{\mathfrak r}_B$, the enveloping algebra of the Lie algebra
$L^+/\mathfrak r$ over $\F_2[\pi]$. If $\bar U=U/\pi U$ and ${\bar
V}={\bar V}/\pi {\bar V}$ we obtain that the canonical map of
${\bar V}$ into $\bar U$ is injective and that $\bar U$ is a free
${\bar V}$-module with basis $\xi_1^{e_1}\cdots\xi_s^{e_s}$
($e_i=0,1$). The algebra $\bar U$ is the enveloping algebra of the
quadratic Lie algebra $\tilde{\mathfrak g}$ and ${\bar V}$ is the
enveloping algebra of the Lie algebra $\tilde{\mathfrak g}^+$ over
$\F_2$.

We now use the fact that $\tilde L/\mathfrak r$, where
$\tilde{\mathfrak r}$ is the image of $\mathfrak r$ in $\tilde L$,
is a strongly free presentation to deduce that $P(\xi)\notin
\tilde{\mathfrak r}$ for every non-zero element $\xi$ of $\tilde
L$ of degree~$1$. Indeed, if $P(\xi)$ lies in $\tilde{\mathfrak
r}$ then, if $\bar\xi$ is the image of $P(\xi)$ in
$\tilde{\mathfrak r}/[\tilde{\mathfrak r},\tilde{\mathfrak r}]$
and $\tilde\xi$ the image of $\xi$ in $\tilde{\mathfrak g}$, we
would have $\bar\xi, \tilde\xi\ne0$
$$
ad(\tilde\xi)(\bar\xi)=0
$$
which contradicts the fact that $\tilde{\mathfrak
r}/[\tilde{\mathfrak r},\tilde{\mathfrak r}]$ is a free $\bar V
$-module via the adjoint representation and the fact that $\bar V$
is an integral domain. Thus multiplication by $P(\xi)=\xi^2$ maps
$\bar V$ injectively in to $\bar V$ which implies that
multiplication by $\xi$ is injective on $\bar V$. This in turn
implies that
$$
P_{\xi\bar V}(t)=tP_{\bar V}(t).
$$
We thus obtain that $P_{\bar U}(t)=(1+t)^sP_{\bar V}(t)$.  This
implies (g) since
$$
P_{\bar V}(t)=\prod_{n\ge2}(1-t^n)^{-b_n}
$$
and $U_{mix}(\gr(G))=U$. The assertion (h) follows form the fact
that $\gr(G)^+$ is a free $\F_2[\pi]$-module and a standard
argument to compute $b_n$ using the formula
$$
(1+t)^r\prod_{n\ge2}(1-t^n)^{-b_n}=\frac{1}{(1-\alpha_1t)\cdots(1-\alpha_st)}
$$

\medskip

\section{Zassenhaus Filtrations}

Theorem~\ref{Cent} can be extended under certain conditions to
filtrations induced by valuations of the completed group ring
$\F_2[[F]]$. The Lie algebras associated to these filtrations are
restricted Lie algebras in the sense of Jacobson~\cite{Ja}. A
sufficient condition is that the initial forms of the relators lie
in a Lie subalgebra over $\F_2$ which is quadratic and that these
initial forms are strongly free. This will give a second proof
that the pro-$2$-group with these relators is of cohomological
dimension $2$.

Let $F$ be the free pro-$2$-group on $x_1,\ldots,x_d$. The
completed group algebra $\bar{\Lambda}=\F_2[[F]]$ over the finite
field $\F_2$ is isomorphic to the algebra of formal power series
in the non-commuting indeterminates $X_1,\ldots,X_d$ over $\F_2$.
Identifying $F$ with its image in $\bar{A}$, we have $x_i=1+X_i$.

If $\tau_1,\ldots,\tau_d$ are integers $>0$, we define a valuation
$\bar{w}$ of $\bar{\Lambda}$ by setting
$$
\bar{w}(\sum_{i_1,\ldots,i_k}a_{i_1,\ldots,i_k}X_{i_1}\cdots
X_{i_k})=\inf_{i_1,\ldots,i_k}(\tau_{i_1}+\cdots+\tau_{i_k}).
$$
Let
$$
\bar{\Lambda}_n=\{u\in \bar{\Lambda}\mid \bar{w}(u)\ge n\},\
\gr_n(\bar{\Lambda})=\bar{\Lambda}_n/\bar{\Lambda}_{n+1},\
\gr(\bar{\Lambda})=\oplus_{n\ge0}\gr_n(\bar{\Lambda}).
$$
Then $\gr(\bar{\Lambda})$ is a graded $\F_2$-algebra. If $\xi_i$
is the image of $X_i$ in $\gr_{\tau_i}(\bar{\Lambda})$ then
$\gr(\bar{\Lambda})$ is the free associative $\F_2$-algebra $\bar
A$ on $\xi_1,\ldots,\xi_d$ with a grading in which $\xi_i$ is of
degree $\tau_i$ . Note that when $\tau_i=1$ for all $i$ we have
$\bar{\Lambda}_n=\bar{I}^n$, where $\bar{I}$ is the augmentation
ideal $(X_1,\ldots,X_d)$ of $\bar{\Lambda}$.

The Lie subalgebra $\bar{L}$ of $\bar A$ generated by the $\xi_i$
is the free Lie algebra over $\F_2$ on $\xi_1,\ldots,\xi_m$ by the
Birkhoff-Witt Theorem. The Lie subalgebra $\tilde L$ generated by
$\xi_1,\ldots,\xi_d$ and the $\xi_i^2$ where $\xi_i$ is of degree
$1$ is the free quadratic Lie algebra on $\xi_1,\ldots,\xi_d$.

A decreasing sequence $(G_n)$ of closed subgroups of a
pro-$2$-group G which satisfies
$$
[G_i,G_j]\subseteq G_{i+j}, \quad G_i^2\subseteq G_{2i}.
$$
is called a called, after Lazard~\cite{Laz}, a $2$-restricted
filtration of $G$.

 For $n\ge1$, let $F_n=(1+\bar{\Lambda}_n)\cap F$. Then
$(F_n)$ is a $2$-restricted filtration of $F$. This filtration is
also called the Zassenhaus $(x,\tau)$-fitration of $F$. The
mapping $x\mapsto x^2$ induces an operator $P$ on $\gr(F)$ sending
$\gr_n(F)$ into $\gr_{2n}(F)$. With this operator, $\gr(F)$ is a
restricted Lie algebra over $\F_2$. If $\tau_i=1$ for all $i$, the
subgroups $F_n$ are the so-called dimension subgroups mod $2$.
They can be defined by
$$
F_n=\langle[y_1,[\cdots[y_{r-1},y_r]\cdots]]^{2^s}\mid
y_1,\ldots,y_r\in F,\ r2^s\ge n\rangle.
$$
\medskip

Let $r_1,\ldots,r_m\in F^2[F,F]$ and let $R=(r_1,\ldots,r_m)$ be
the closed normal subgroup of $F$ generated by $r_1,\ldots,r_m$.
Let $\rho_i\in\gr_{h_i}(F)$ be the initial form of $r_i$ with
respect to the Zassenhaus $(x,\tau)$-filtration $(F_n)$ of $F$. If
$G=F/R$ and $G_n$ is the image of $F_n$ in $G=F/R$ then
$(G_n)_{n\ge1}$ is a $2$-restricted filtration of $G$. Let
$\bar{\Gamma}_n$ be the image of $\bar{\Lambda}_n$ in
$\bar{\Gamma}=\F_2[[G]]$.
\medskip

\begin{theorem}\label{Zas}
Suppose that the initial forms $\rho_1,\ldots,\rho_m$ of
$r_1,\ldots,r_m$ are in $\tilde L$ and are strongly free. Then
\begin{enumerate}[{\rm (a)}]
\item We have $\gr(G)=\gr(F)/(\rho_1,\ldots,\rho_m)$,
\item The group $R/R^2[R,R]$ is a free $\F_2[[G]]$-module on the images of
$r_1,\ldots,r_m$,
\item The presentation $G=F/R$ is minimal and $\cd(G)=2$.
\item The enveloping algebra of $\gr(G)$ is the graded algebra
associated to the filtration $(\bar{\Gamma}_n)$.
\item The filtration $(\bar{\Gamma}_n)$ of $\bar{\Gamma}$ induces the filtration
$(G_n)$ of $G$.
\item The Poincar\'e series of $\gr(\bar{\Gamma})$ is
$1/(1-(t^{\tau_1}+\cdots+t^{\tau_d})+t^{h_1}+\ldots+t^{h_m})$.
\item If $\tau_i=1$ for all $i$ and $a_n=\dim \gr_n(G)$ then
$$
\prod_{n\ge1}(1+t^n)^{a_n}=\frac{1}{1-dt+mt^2}.
$$
\end{enumerate}
\end{theorem}
\begin{proof} In \cite{Ko3}, Koch proves that if $\bar{\mathcal
R}/\bar{\mathcal R}\bar I$ is a free $\bar A/\bar{\mathcal R}$
module on the images of $\rho_1,\ldots,\rho_m$ then
$\gr(\bar{\Gamma})=\bar A /\bar{\mathcal R}$, where $\bar{\mathcal
R}$ is the ideal of $\bar A=\gr(\bar{\Lambda})$ generated by
$\rho_1,\ldots,\rho_m$. The former is true if
$\rho_1,\ldots,\rho_m$ lie in $\tilde L$ and are strongly free
since $\bar{\mathcal R}/\bar{\mathcal R}\bar I$ is the image of
the free $\bar A/\bar{\mathcal R}$-module ${\tilde{\mathfrak
r}}/[\tilde{\mathfrak r},\tilde{\mathfrak r}]$ under the injective
mapping
$$
{\tilde{\mathfrak r}}/[\tilde{\mathfrak r},\tilde{\mathfrak
r}]\rightarrow\bar I/\bar{\mathcal R}\bar I
$$
 where $\tilde{\mathfrak r}$ is the ideal of the quadratic
Lie algebra $\tilde L$ generated by $\rho_1,\ldots,\rho_m$. Now
consider the exact sequence
$$
0\rightarrow \mathfrak r/[\mathfrak r,\mathfrak r]\rightarrow
\gr(\bar{\Gamma})^d\rightarrow\gr(\bar{\Gamma})\rightarrow\F_2\rightarrow0,
$$
 Since $\mathfrak r/[\mathfrak r,\mathfrak r]$ is a free
$\gr(\bar{\Gamma})$-module of rank $m$, we obtain the exact
sequence
$$
0\rightarrow \gr(\bar{\Gamma})^m\rightarrow
\gr(\bar{\Gamma})^d\rightarrow\gr(\bar{\Gamma})\rightarrow\F_2\rightarrow0.
$$
This yields (f). By a result of Serre (cf. \cite{Laz}, V, 2.1), we
obtain the exact sequence
$$
0\rightarrow \bar{\Gamma}^m\rightarrow
\bar{\Gamma}^d\rightarrow\bar{\Gamma}\rightarrow\F_2\rightarrow0.
$$
By a result of \cite{Br}, section 5, this proves (b) and (c). If
$\mathfrak R=(\rho_1,\ldots,\rho_m)$ is the ideal of the
restricted Lie algebra $\gr(F(X))$ generated by
$\rho_1,\ldots,\rho_m$, we have canonical homomorphisms of
restricted Lie algebras
$$
\gr(F(X))/\mathfrak
R\rightarrow\gr(G)\rightarrow\gr'(G)\rightarrow\gr(\bar{\Gamma}),
$$
where the first arrow is surjective and $\gr'(G)$ is the
restricted Lie algebra associated to the Zassenhaus filtration
$(G_n')$ of $G$ induced by the filtration of $\Gamma$. Since
$\gr(\bar{\Gamma})$ is the enveloping algebra of the restricted
Lie algebra $\gr(F)/\mathfrak R$, the Birkhoff-Witt Theorem for
restricted Lie algebras shows that all arrows are injective which
yields (a) and (d). The injectivity of $\gr(G)\rightarrow\gr'(G)$
yields $G_n=G_n'$ for all $n$ by induction which proves (e). The
proof of (g) follows from (d), (e), (f) and Proposition A3.10 of
\cite{Laz}.
\end{proof}
\medskip

\noindent{\bf Remark.} The formula given in (g) partially answers
a question of Morishita stated in \cite{Mo} in a remark after
Theorem 3.6.

\section{Proof of Theorem~\ref{main1}}

Let $(\chi_i)_{1\le i\le d}$ be a basis of $H^1(G)$ with
$(\chi_i)_{i\in S}$ a basis of $U$ and $(\chi_j)_{j\in S'}$ a
basis of $V$. Let $(\xi_i)$ be the dual basis of $H^1(G)^*=L_1(G)$
and let $g_i$ be any lift of $\xi_i$ to $G$. Let $F$ be the free
pro-$2$-group on $x_1,\ldots,x_d$ and let $f:F\rightarrow G$ be
the homomorphism sending $x_i$ to $g_i$. Then the induced mapping
of $L_1(F)$ into $L_1(G)$ is an isomorphism which we use to
identify these two groups. If $R$ is the kernel of $f$ the
presentation $G=F/R$ is minimal and the transgression map ${\rm
tg} : H^1(R/R^2[R,F])\rightarrow H^2(G)$ is an isomorphism. Hence
$\tg^* :H^2(G)^*\rightarrow R/R^2[R,F]$ is an isomorphism which we
use to identify these two groups. If $\psi$ is the inverse of
$\tg^*$ and $r\in R$ we let $\bar r=\psi(r)$.

The cup product $H^1(G)\otimes H^1(G)\rightarrow H^2(G)$ vanishes
on the subspace $W$ generated by elements of the form $a\otimes
b+b\otimes a$ and so, by duality, induces a homomorphism

$$
H^2(G)^*\rightarrow L_1(F)\otimes L_1(F)=(H^1(G)\otimes H^1(G))^*,
$$
whose image is contained in $W^0$, the annihilator of the subspace
$W$. Since $\dim W=d(d-1)/2$ we have $\dim W^0=d(d+1)/2$. Now
$L_2(F)$ can be identified with the subspace of the tensor algebra
of $L_1(F)$ generated by the elements of the form $\xi^2$ and
$[\xi,\eta]=\xi\eta+\eta\xi$. Since these elements lie in $W^0$
and $\dim L_2(F)=\dim W^0$ we obtain that $W^0=L_2(F)$. If
$$
H^1(G)\otimes' H^1(G)=(H^1(G)\otimes H^1(G))/W
$$
is the symmetric tensor product of $H^1(G)$ with itself we have
$$
H^1(G)\otimes' H^1(G)=U\otimes'U\oplus V\otimes'V\oplus
U\otimes'V,
$$
where $U\otimes'V$ is the image of $U\otimes'V$ in $H^1(G)\otimes'
H^1(G)$. Since the cup-product vanishes on $U\otimes'U$ it induces
a homomorphism
$$
\varphi : V\otimes'V\oplus U\otimes'V=(H^1(G)\otimes'
H^1(G))/U\otimes'U\rightarrow H^2(G)
$$
which is surjective since, by assumption, the cup-product maps
$U\otimes V$ onto $H^2(G)$. Since the annihilator of $U\otimes'U$
is contained in $\mathfrak a_2$, where $\mathfrak a$ is the ideal
of $L(F)$ generated by the $\xi_i$ with $i\in S'$, we get an
injective homomorphism
$$
\varphi^* : H^2(G)^*\rightarrow \mathfrak a_2.
$$

Let $r_1,\ldots,r_m$ generate $R$ as a closed normal subgroup of
$F$.  Since $r_i\in F_2$ we have
$$
r_k\equiv
\prod_{i=1}^dx_i^{2a_{ik}}\prod_{i<j}[x_i,x_j]^{a_{ijk}}\quad {\rm
mod\ } F_3
$$
with $a_{ik}=\bar r_k(\chi_i\cup\chi_i)$ and $a_{ijk}=\bar
r_k(\chi_i\cup\chi_j)$ (cf. \cite{L}, Prop. 3). Moreover, if
$\rho_k$ is the initial form of $r_k$, we have
$$
\varphi^*(\bar
r_k)=\rho_k=\sum_{i=1}^da_{ik}\xi_i^2+\sum_{i<j}a_{ijk}[\xi_i,\xi_j].
$$
By Theorems~\ref{sfred} and \ref{SF}, the elements
$\rho_1,\ldots,\rho_m$ form a strongly free sequence if their
images in $(\mathfrak a/\mathfrak a^*)_2=\mathfrak a_2/\mathfrak
b$, where $\mathfrak b$ is the subspace of $\mathfrak a_2$
generated by the elements $\xi_i^2$, $[\xi_i,\xi_j]$ with $i,j\in
S'$, are linearly independent. If $\mathfrak c$ is the subspace of
$\mathfrak a_2$ generated by the elements $[\xi_i,\xi_j]$ with
$i\in S,j\in S'$ then  $\mathfrak a_2=\mathfrak b \oplus\mathfrak
c$. The images of the $\rho_i$ in $\mathfrak a_2/\mathfrak b$ form
a linearly independent sequence if and only if the projections of
the $\rho_i$ on $\mathfrak c$ form an independent sequence. But
this is equivalent to the composite
$$
H^2(G)^*\rightarrow \mathfrak a_2\rightarrow\mathfrak c
$$
being injective. Now $\mathfrak a_2$ is the dual space of
$$
(H^1(G)\otimes'H^1(G))/U\otimes'U=V\otimes'V\oplus U\otimes'V
$$
and, with respect to this duality, we have $\mathfrak
c=(V\otimes'V)^0$ which implies that the canonical injection
$$
\iota : U\otimes'V\rightarrow V\otimes'V\oplus U\otimes'V
$$
is dual to the projection of $\mathfrak a_2$ onto $\mathfrak c$.
Since $\phi\circ\iota$ is surjective it dual
$\iota^*\circ\varphi^*$ is injective. But the latter is the
composite $H^2(G)^*\rightarrow \mathfrak a_2\rightarrow\mathfrak
c$.

\section{Proof of Theorem~\ref{main2} and Examples}

Without loss of generality, we may assume $S_0=\{q_1,\ldots,q_m\}$
with $m\ge2$, $q_1\equiv1$ mod $4$ and $q_m\equiv3$ mod $4$. Let
$q_1',\ldots,q_m'$ be primes $\equiv1$ mod $4$ which are not in
$S_0$ and such that
\begin{enumerate}[(a)]
\item $q_i'$ is a square mod $q_j'$ for all $i,j$,
\item $q_1'$ is not a square mod $q_m$ and $q_i'$ is not a square mod $q_i$ and $q_{i-1}$ for
$1<i\le m$.
\end{enumerate}

Let $S=\{q_1',q_1,q_2',q_2,\ldots,q_m',q_m,q_{m+1}\}$ where
$q_{m+1}$is a prime $\equiv3$ mod $4$ distinct from
$q_1,\ldots,q_m$ and such that $q_{m+1}$ is not a square mod
$q_1'$ but is a square mod $q_i'$ for all $i\ne1$. Let
$$
(p_1,\ldots,p_{2m+1})=(q_1',q_1,q_2',q_2,\ldots,q_m',q_m,q_{m+1}).
$$
and let $x_1,\ldots,x_{2m+1}$ be generators for the inertia
subgroups of $G_S(2)$ at the primes $p_1,\ldots,p_{2m+1}$
respectively. Then, by \cite{Ko}, Theorem 11.10 and Example~11.12,
the group $G=G_S(2)$ has the presentation $G=F(X)/R=\langle
x_1,\ldots,x_{2m+1}\mid r_1,\ldots,r_{2m+1},r\rangle$, where
\begin{align*}
r_i&\equiv x_i^{2a_i}\prod_{j=1}^{2m+1}[x_i,x_j]^{\ell_{ij}}\ {\rm
mod}\ F_3,\\
r&\equiv\prod_{i=1}^{2m+1}x_i^{a_i}\ {\rm mod}\ F_2
\end{align*}
with $a_i=0$ if and only if $p_i\equiv1$ mod $4$ and $\ell_{ij}=1$
if $p_i$ is not a square mod $p_j$ and $0$ otherwise. Moreover, we
can omit the relator $r_{2m+1}$. By construction we have
$$
r\equiv\prod_{i=2}^{m-1}x_{2i}^{a_{2i}}x_{2m}x_{2m+1}\ {\rm mod}\
F_2
$$
so that $x_{2m+1}\equiv x_{2m}x_4^{a_4}\cdots x_{2m-2}^{a_{2m-2}}$
mod $F_2$. Hence $G=\langle x_1,\ldots,x_{2m}\mid
r_1',\ldots,r_{2m}'\rangle$ where
$$
r_i'\equiv x_i^{2a_i}\prod_{j=1}^{2m}[x_i,x_j]^{\ell_{ij}'}
$$
with $\ell_{ij}'=0$ if $i,j$ are odd and
$$
\ell_{12}'=\ell_{23}'=\ell_{34}'=\cdots=\ell_{2m-1,2m}'=\ell_{2m,1}'=1
$$
but $\ell_{1,2m}'=0$. The image of the initial form of $r_i'$ in
${\tilde L}_{\rm mix}(X)$ (here $\tau_i=1$ for all $i$) is
$$
\rho_i'=\xi_i^{2a_i}+\sum_{j=1}^{2m}\ell_{ij}'[\xi_i,\xi_j].
$$
By Corollary~\ref{sf} the sequence $\rho_1',\ldots,\rho_{2m}'$ is
strongly free in ${\tilde L}_{\rm mix}(X)$ and therefore $G$ is
mild by Theorem~\ref{sfred}.
\medskip

\noindent{\bf Example 1.} To illustrate the above proof, let
$S_0=\{13,3\}=\{q_1,q_2\}$. Then $q_1'=41$, $q_2'=5$, $q_3=19$
satisfy the required conditions. Then
$$
S=\{41,13,5,3,19\}=\{p_1,p_2,p_3,p_4,p_5\}
$$
and the relators for the first presentation are
\begin{align*}
r_1&\equiv[x_1,x_2][x_1,x_4][x_1,x_5]\ {\rm mod}\ F_3,\\
r_2&\equiv [x_2,x_1][x_2,x_3][x_2,x_5]\ {\rm mod}\ F_3,\\
r_3&\equiv [x_3,x_2][x_3,x_4]\ {\rm mod}\ F_3,\\
r_4&\equiv x_4^2[x_4,x_1][x_4,x_3][x_4,x_5]\ {\rm mod}\ F_3,\\
r_5&\equiv x_5^2[x_5,x_1][x_5,x_2]\ {\rm mod}\ F_3,\\
r&=x_4x_5\ {\rm mod}\ F_2.
\end{align*}
Hence $G=G_S(2)$ has the presentation $<x_1,x_2,x_3,x_4\mid
r_1',r_2',r_3',r_4'>$ where
\begin{align*}
r_1'&\equiv[x_1,x_2]\ {\rm mod}\ F_3,\\
r_2'&\equiv[x_2,x_1][x_2,x_3][x_2,x_4]\ {\rm mod}\ F_3,\\
r_3'&\equiv[x_3,x_2][x_3,x_4]\ {\rm mod}\ F_3,\\
r_4'&\equiv x_4^2[x_4,x_1][x_4,x_3]\ {\rm mod}\ F_3.
\end{align*}
\medskip

\noindent{\bf Example 2.} This example is due to Denis Vogel and
while it does not illustrate exactly the above proof it does
contain the basic idea which led to the result. Let
$S=\{5,29,7,11,3\}$. Using the above notation for a Koch
presentation of $G_S(2)$ with $p_1=5,p_2=29,p_3=7,p_4=11,p_5=3$ we
have
\begin{align*}
r_1&\equiv[x_1,x_3][x_1,x_5]\ {\rm mod}\ F_3,\\
r_2&\equiv[x_2,x_4][x_2,x_5]\ {\rm mod}\ F_3,\\
r_3&\equiv x_3^2[x_3,x_1][x_3,x_4]\ {\rm mod}\ F_3,\\
r_4&\equiv x_4^2[x_4,x_2][x_4,x_5]\ {\rm mod}\ F_3,\\
r_5&\equiv x_5^2[x_5,x_1][x_5,x_2]\ {\rm mod}\ F_3,\\
r&\equiv x_3x_4x_5\ {\rm mod}\ F_2.
\end{align*}
Omitting $r_5$ and setting $x_5=x_3x_4$ mod $F_2$, we get
\begin{align*}
r_1'&\equiv[x_1,x_4]\ {\rm mod}\ F_3,\\
r_2'&\equiv[x_2,x_3]\ {\rm mod}\ F_3,\\
r_3'&\equiv x_3^2[x_3,x_1][x_3,x_4]\ {\rm mod}\ F_3,\\
r_4'&\equiv x_4^2[x_4,x_2][x_4,x_3]\ {\rm mod}\ F_3.
\end{align*}
The images of the initial forms of these relators in ${\tilde
L}_{\rm mix}(X)$ (all $\tau_i=1$) are
\begin{align*}
\rho_1'&=[\xi_1,\xi_4],\\
\rho_2'&=[\xi_2,\xi_3],\\
\rho_3'&=\xi_3^2+[\xi_3,\xi_1]+[\xi_3,\xi_4],\\
\rho_4'&=\xi_4^2+[\xi_4,\xi_2]+[\xi_4,\xi_3].
\end{align*}
If $\mathfrak a$ is the ideal of ${\tilde L}_{\rm mix}(X)$
generated by $\xi_3,\xi_4$ the $\rho_i'$ are in $\mathfrak a$ and
their images in $\mathfrak a/\mathfrak a^*$ are the classes of
$$
[\xi_1,\xi_4],\ [\xi_2,\xi_3],\ [\xi_1,\xi_3],\ [\xi_2,\xi_4]
$$
which are part of a basis for $(\mathfrak a/\mathfrak a^*)_2$.
Hence $G_S(2)$ is mild. If $a_n=\dim L(G_S)$ then $a_1=4$ and
$$
a_n=\sum_{k=2}^n(\frac{1}{k}\sum_{\ell|k}\mu(\frac{k}{\ell})(2^{\ell+1}+(-1)^\ell4)
$$ for $n\ge2$ by Theorem~\ref{Cent}.

\end{document}